\documentclass[12pt]{amsart}
\usepackage[utf8]{inputenc}
\usepackage{amsmath, amsthm, amssymb, amsfonts}
\usepackage{mathrsfs}
\usepackage{hyperref}
\usepackage{enumitem}
\usepackage{geometry}
\usepackage{xcolor}
\usepackage{tikz-cd}
\newcommand{\bb}{\mathbb}
\newcommand{\al}{\mathcal}

\newcommand{\Lk}{{\rm Lk}}
\newcommand{\diam}{{\rm diam}}

\newtheorem{thm}{Theorem}[section]
\newtheorem{lem}[thm]{Lemma}
\newtheorem{prop}[thm]{Proposition}
\newtheorem{defn}[thm]{Definition}
\newtheorem{cor}[thm]{Corollary}
\newtheorem{example}[thm]{Example}
\newtheorem{rem}[thm]{Remark}

\newtheorem{Observation}[thm]{Observation}

\newcommand{\vol}{{\rm vol}}
\newcommand{\dist}{{\rm dist}}

\newcommand{\length}{{\rm length}}
\newcommand{\Lip}{{\rm Lip}}

\title{Bi-H\"older invariants in o-minimal structures}

\author[An V. Q. Huynh]{An V. Q. Huynh}
\address{FPT University, Danang, Vietnam}
\email{huynhvanquocan@gmail.com}

\author[Minh B. Nguyen]{Minh B. Nguyen}
\address{Hanoi National University of Education, Vietnam}
\email{ngbinhminh.dsmun@gmail.com}

\author[Nhan X. V. Nguyen]{Nhan X. V. Nguyen}
\address{FPT University, Danang, Vietnam}
\email{nguyenxuanvietnhan@gmail.com}

\author[Minh Q. Vu]{Minh Q. Vu}
\address{University of Science and Technology of Hanoi, Vietnam}
\email{vuq013588@gmail.com}

\subjclass[2010]{ 32S50, 14B05, 32C05, 32C15}

\begin{document}

\maketitle

\begin{abstract}  We prove that for any two  definable germs in a polynomially bounded o-minimal structure, there exists a critical threshold $\alpha_0 \in (0, 1)$ such that if these germs are bi-$\alpha$-H\"older equivalent for some $\alpha \geq \alpha_0$, then they  satisfy the following:
\begin{itemize}[label=$\circ$]
   \item The Lipschitz normal embedding (LNE) property is preserved; that is, if one germ is LNE then  so is the other;
\item Their tangent cones have the same dimension;
\item The links of their tangent cones have isomorphic homotopy groups.
\end{itemize}
As an application, we give a simple proof that a complex analytic germ which is bi-$\alpha$-H\"older homeomorphic to the germ of a Euclidean space for some $\alpha$ sufficiently close to $1$ must be smooth.
This provides a slightly
stronger version of Sampaio’s smoothness theorem, in which the germs are assumed to be bi-$\alpha$-H\"older homeomorphic for every $\alpha \in (0,1)$.
\end{abstract}
\section{Introduction}

A crucial problem in Singularity Theory is to determine which properties of a singular germ are preserved under various notions of metric equivalence. A central object of study is the tangent cone, which serves as a first-order approximation of the germ at the origin. Numerous fundamental invariants associated with tangent cones have been shown to be preserved under bi-Lipschitz homeomorphisms. For instance, Koike and Paunescu~\cite{Koike2009} proved that the dimension of the tangent cone of a subanalytic germ is a bi-Lipschitz invariant. This result was later extended to definable sets in o-minimal structures by Koike, Loi, Paunescu, and Shiota~\cite{Koike2013}. A major advance was achieved by Sampaio~\cite{Sampaio2016}, who showed that the tangent cone of a subanalytic germ itself is invariant under bi-Lipschitz homeomorphisms (with respect to the outer metric). In a related direction, Fernandes and Sampaio~\cite{FS2019} established that Lipschitz normal embedding (LNE) in an o-minimal structure is inherited by the tangent cone. More recently, the third author~\cite{Nhan2024}  provided a necessary and sufficient condition for two definable germs to have bi-Lipschitz homeomorphic tangent cones.

Nevertheless, the bi-Lipschitz condition (corresponding to $\alpha=1$ in H\"older terminology) is often too restrictive for many purposes. A natural relaxation is to consider bi-$\alpha$-H\"older equivalence for $\alpha \in (0,1)$, which provides metric equivalences interpolating between the topological category and the bi-Lipschitz one.

Recent work reveals that this intermediate has both rigid and flexible aspects. For complex analytic plane curves, Fernandes, Sampaio, and Silva~\cite{FEJ2018} proved that  strongly bi-$\alpha$-H\"older equivalence (being bi-$\alpha$-H\"older homeomorphic for all $\alpha\in (0,1)$) determines  the characteristic exponents of each branch as well as their pairwise intersection multiplicities. In the real analytic setting, by contrast, Fernandes et al. ~\cite{FSSS2024} showed that analytic function germs admit moduli under bi-$\alpha$-H\"older right equivalence for any fixed $\alpha \in (0,1)$, highlighting the flexibility.

More recently, Sampaio~\cite{Sampaio2025} established a interesting rigidity phenomenon in the complex category: if a complex analytic germ $(X,0) \subset \mathbb{C}^n$ is strongly bi-$\alpha$-H\"older equivalent to the germ $(\mathbb{R}^k,0)$ of a Euclidean plane, then $(X,0)$ must be smooth.

Motivated by these developments, we consider the following natural questions, listed in decreasing order of difficulty. Let $(X, 0)$ and $(Y, 0)$ be definable germs in an o-minimal structure and suppose they are strongly bi-$\alpha$-H\"older equivalent:

\begin{enumerate}
\item Are $(X,0)$ and $(Y,0)$  bi-Lipschitz equivalent?
\item Are their tangent cones $C_0(X)$ and $C_0(Y)$ bi-Lipschitz equivalent?
\item Are $C_0(X)$ and $C_0(Y)$ topologically equivalent?
\item Do $C_0(X)$ and $C_0(Y)$ have the same dimension?
\item If $(X,0)$ is Lipschitz normally embedded, must $(Y,0)$ also be Lipschitz normally embedded?
\end{enumerate}

In general, the answer to the above questions is negative. Counterexamples are presented in Section~\ref{section6}; notably, all such examples live in non-polynomially bounded o-minimal structures. This naturally leads to the question of whether positive answers hold in the polynomially bounded case, where the Łojasiewicz inequality imposes strong control on asymptotic distortion.

Our main results provide affirmative answers to Questions~(4) and~(5), as well as a partial affirmative answer to Question~(3), under the assumption that the underlying o-minimal structure is polynomially bounded. More precisely, we prove that given definable germs  $(X,0)$ and $(Y,0)$ in a polynomially bounded o-minimal structure, there exists a threshold $\alpha_0 \in (0,1)$ (depending only on $X$ and $Y$) such that if $(X,0)$ and $(Y,0)$ are bi-$\alpha$-H\"older equivalent for some $\alpha \ge \alpha_0$, then
\begin{enumerate}[label=(\roman*)]
\item if $(X,0)$ is Lipschitz normally embedded, then so is $(Y,0)$ \quad (Theorem~\ref{thm_lne});
\item $\dim C_0(X) = \dim C_0(Y)$ \quad (Theorem~\ref{thm_dim2});
\item the links of the tangent cones $C_0(X)$ and $C_0(Y)$ have isomorphic homotopy groups \quad (Theorem~\ref{thm_homotopy2}).
\end{enumerate}

Statement (ii) significantly strengthens the classical bi-Lipschitz dimension invariance result of Koike and Paunescu~\cite{Koike2009} by replacing bi-Lipschitz equivalence with bi-$\alpha$-H\"older equivalence for $\alpha$ sufficiently close to $1$.

As an immediate application of (i) and (iii), we obtain a simple proof showing that if a complex analytic 
germ $(X,0)$ is bi-$\alpha$-H\"older equivalent to the germ of a $k$-plane for some $\alpha$ close 
to~$1$, then $(X,0)$ must be smooth. This yields a strengthening of Sampaio’s smoothness theorem \cite{Sampaio2025}.

\subsection*{Notation}

Throughout this paper we adopt the following notation:

\begin{itemize}
\item $B(x,r)$ and $S(x,r)$ denote the open ball and the sphere of radius $r$ centered at $x \in \mathbb{R}^n$, respectively. When the center is the origin and the radius is $1$, we simply write $S^{n-1} := S(0,1)$.

\item Let $X \subset \mathbb{R}^n$. We denote by $\diam(X)$ the diameter of $X$.
For a point $x \in \mathbb{R}^n$, we denote by $\dist(x,X)$ the Euclidean distance from $x$ to $X$.

\item For a definable set $X \subset \mathbb{R}^n$, $\overline{X}$ denotes its Euclidean closure. A \emph{closed definable germ at $x_0$}  means the germ at $x_0$ of a closed definable set. If $X$ is a definable germ at $x_0 \in \mathbb{R}^n$, the \emph{link of $X$ at $x_0$} is the set 
  \[
  \Lk_r(X,x_0) := X \cap S(x_0,r).
  \] for $r$ sufficiently small. 
  The link is well-defined up to definable homeomorphism. When the choice of radius is irrelevant, we write $\Lk(X,x_0)$.

\item For a definable set $K \subset \mathbb{R}^n$, we denote its homotopy groups by 
$\pi_k(K)$ for $k \ge 0$.  Note that $\pi_0(K)$ is not a group, but simply the set of path-connected components of $K$. When no ambiguity arises, we write $\pi_*(K)$ for the collection of all homotopy groups. We say that two sets have the same homotopy type if their corresponding homotopy groups are isomorphic.

\item $\vol_k$ ($k = 0,\dots,n$) denotes the $k$-dimensional Lebesgue (Hausdorff) measure on $\mathbb{R}^n$. Note that every definable set is Lebesgue measurable.

\item For functions $f,g : U \to [0,\infty)$ defined on a set $U \subset \mathbb{R}^n$, we write
  \[
  f \lesssim g \quad \text{if} \quad \exists\, C > 0 \quad \text{such that} \quad f(x) \le C\, g(x) \quad \forall\, x \in U.
  \]
  The constant $C$ is called \emph{a constant for the relation $\lesssim$}. We write $f \sim g$ if $f \lesssim g$ and $g \lesssim f$.
\end{itemize}

\subsection*{Acknowledgements}

Part of this work was carried out during the Research Experiences for Undergraduates (REU) programme organised by the Vietnam Institute for Advanced Study in Mathematics (VIASM). The authors gratefully acknowledge VIASM’s warm hospitality and support. The research of the third author is funded by the Vietnam National Foundation for Science and Technology Development (NAFOSTED) under grant number 101.04-2024.09.

We thank an anonymous referee for pointing out an error in the proof of Theorem 3.1 in an earlier version of this paper.

\section{Preliminaries}\label{section2}
\subsection{Bi-$\alpha$-H\"older equivalence}

\begin{defn}[$\alpha$-H\"older map]\rm 
Let $(X, d_X)$ and $(Y, d_Y)$ be metric spaces, and let $\alpha \in (0,1]$. A map $f : X \to Y$ is called \emph{$\alpha$-H\"older} if there exists a constant $L > 0$ such that
\[
d_Y\big(f(x), f(x')\big) \le L \, d_X(x, x')^{\alpha}
\quad \text{for all } x, x' \in X.
\]
The constant $L$ is referred to as a \emph{H\"older constant} of $f$.
\end{defn}

In this paper, we work with subsets of $\bb R^n$ equipped with the metric induced by the Euclidean metric on \( \mathbb{R}^n \).

\begin{defn}[bi-$\alpha$-H\"older equivalence]\rm 
Let $0<\alpha\le 1$. Two germs $(X,x_0)\subset\mathbb{R}^n $ and $(Y,y_0)\subset\mathbb{R}^m$ are said to be \textit{bi-$\alpha$-H\"older equivalent}  (or \textit{bi-$\alpha$-H\"older homeomorphic}) if there exists a homeomorphism $\varphi:(X,x_0)\to(Y,y_0)$ and a constant $L\ge1$ such that
$$
\frac{1}{L}\|x_1-x_2\|^{1/\alpha}\leq \|\varphi(x_1)-\varphi(x_2)\|\leq L\|x_1-x_2\|^{\alpha}
$$
for all $x_1,x_2$ sufficiently close to $x_0$. 
\end{defn}

\begin{defn}[embedded bi-$\alpha$-H\"older equivalence]\rm 
Let $0<\alpha\le 1$. Two germs $(X,x_0)$ and $(Y,y_0)$ in $\mathbb{R}^n$ are said to be \textit{embedded bi-$\alpha$-H\"older equivalent}  (or \textit{embedded bi-$\alpha$-H\"older homeomorphic}) if there exist neighborhoods $U$ of $x_0$  and $V$ of $y_0$ in $\bb R^n$ and a homeomorphism $\varphi:(U,x_0)\to(V,y_0)$ and a constant $L\ge1$ such that $\varphi(U\cap X) = V \cap Y$ and 
$$
\frac{1}{L}\|x_1-x_2\|^{1/\alpha}\leq \|\varphi(x_1)-\varphi(x_2)\|\le L\|x_1-x_2\|^{\alpha}
$$
for all $x_1,x_2$ in $U$. 
\end{defn}

\begin{rem}\rm 
\begin{itemize}
    \item When $\alpha = 1$, these notions coincide with the standard bi-Lipschitz equivalences. 
For $\alpha < 1$, they provide a metric notation lying between topological and bi-Lipschitz equivalence.
\item If two germs $(X, x_0)$ and $(Y, y_0)$ are bi-$\alpha$-H\"older equivalent, then they  are automatically  bi-$\beta$-H\"older equivalent for every $0<\beta\leq \alpha$.
\end{itemize}

\end{rem}

\subsection{O-minimal structures}

In this section, we recall the definition of an o-minimal structure on $(\mathbb{R}, +, \cdot)$
and summarize some basic results that will be used later. 
We refer the reader to \cite{Coste, Dries, Loi1} for detailed expositions on the theory of o-minimal structures. 

\begin{defn}\rm 
A \emph{structure} on $(\mathbb{R}, +, \cdot)$ is a sequence 
$\mathcal{D} = (\mathcal{D}_n)_{n \in \mathbb{N}}$ such that, for each $n \ge 1$, the following properties hold:
\begin{enumerate}
\item[(1)] $\mathcal{D}_n$ is a Boolean algebra of subsets of $\mathbb{R}^n$;
\item[(2)] If $A \in \mathcal{D}_n$ then $A \times \bb R$ and $\bb R\times A$ are in $\mathcal{D}_{n+1}$;
\item[(3)] For every polynomial $P \in \mathbb{R}[X_1, \ldots, X_n]$, the zero set 
$\{x \in \mathbb{R}^n : P(x) = 0\}$ belongs to $\mathcal{D}_n$.
\item[(4)] If $A \in \mathcal{D}_n$, then the projection 
$\pi(A) \in \mathcal{D}_{n-1}$, where $\pi : \mathbb{R}^n \to \mathbb{R}^{n-1}$ 
is the projection onto the first $(n-1)$ coordinates.
\end{enumerate}
Such a structure $\mathcal{D}$ is said to be \emph{o-minimal} if, in addition,
\begin{enumerate}
\item[(5)] Every set $A \in \mathcal{D}_1$ is a finite union of points and open intervals.
\end{enumerate}
A set in $\al D_n$ for some $n$ is called a \emph{definable} set. A map $f: A \to \bb R^m$ is called definable if its graph is definable. 
\end{defn}

\medskip

\begin{defn}\rm
An o-minimal structure $\mathcal{D}$ is said to be \emph{polynomially bounded} 
if every definable function $f : \mathbb{R} \to \mathbb{R}$ satisfies 
$|f(x)| \le x^N$ for some integer $N$ and all sufficiently large $x$. 
\end{defn}

It has been known by Miller’s Dichotomy Theorem \cite{Miller1993}, every o-minimal structure on $(\mathbb{R}, +, \cdot)$ is either polynomially bounded or exponential (i.e., it contains the graph of $\exp(x)$). The class of semialgebraic sets ($\bb R_{\rm{alg}}$)  and the class of globally subanalytic sets ($\bb R_{\rm{an}}$) are typical examples of polynomially bounded o-minimal structures. On the other hand, the structure generated by the exponential function ($\bb R_{\rm exp})$ (see \cite{Wilkie96}) or the one generated by the exponential function and restricted analytic functions ($\bb R_{\rm{an, exp}}$) are examples of non-polynomially bounded o-minimal structures (see \cite{DM1994})

\medskip

\begin{defn}\rm 
Let $p : X \to \mathbb{R}^k$ be a continuous definable map, where $X \subset \mathbb{R}^n$ is definable and $n \ge k$.
We say that $p$ is \emph{definably trivial over} a definable set $Y \subset \mathbb{R}^k$
if there exist a definable set $F$ and a definable homeomorphism
\[
h : p^{-1}(Y) \longrightarrow F \times Y
\]
such that $p_1 \circ h = p|_{p^{-1}(Y)}$, where $p_1 : F \times Y \to Y$ is the projection onto the second factor.
If $Z \subset X$ is a definable subset, we say that $h$ is \emph{compatible with} $Z$
if there exists a definable subset $G \subset F$ such that 
$h(Z \cap p^{-1}(Y)) = G \times Y$. The homeomorphism $h$  is called a \textit{trivialization} of $p$ over $Y$.
\end{defn}

\begin{thm}[Hardt’s Triviality Theorem {\cite{Coste}}]
Let $p : X \subset \mathbb{R}^n \to \mathbb{R}^k$ be a continuous definable map,
and let $Z_1, \ldots, Z_m \subset X$ be definable subsets.
Then there exists a finite definable partition $\{Y_i\}_{i=1}^N$ of $p(X)$ such that for each $i$, the restriction $p|_{p^{-1}(Y_i)}$ is definably trivial over $Y_i$, and is compatible with all  $Z_1, \ldots, Z_m$.
\end{thm}

\begin{Observation}\label{Observation1}\rm
\begin{enumerate}
    \item For any $y, y' \in Y_i$, the fibers $p^{-1}(y)$ and $p^{-1}(y')$ are definably homeomorphic.
    Moreover, one may take the fiber $F$ in the trivialization to be $p^{-1}(y)$ for any fixed $y \in Y_i$.

    \item If $k = 1$ and $Y_i = (a, b)$, then for every $\varepsilon_1 < \varepsilon_2$ in $(a, b)$, 
    the trivialization $h$ induces  strong deformation retracts from $p^{-1}([c, \varepsilon_2])$ onto 
    $p^{-1}([c, \varepsilon_1])$ for every $c \leq  \varepsilon_1$ and from $p^{-1}([\varepsilon_1, d])$ onto $p^{-1}([\varepsilon_2, d])$ for every $d \geq \varepsilon_2$. To see this, begin with the strong deformation retract
    $\Phi : [0,1] \times [c, \varepsilon_2] \to [c, \varepsilon_2]$ from $[c, \varepsilon_2]$ onto 
    $[c, \varepsilon_1]$ defined by
    \[
    \Phi_s(t) = \Phi(s,t) :=
    \begin{cases}
        t, & \text{if } t \le \varepsilon_1, \\[4pt]
        (1-s)t + s\varepsilon_1, & \text{if } t > \varepsilon_1 .
    \end{cases}
    \]
    Then the map 
    \[
    H : [0,1] \times \rho^{-1}([c, \varepsilon_2]) \to \rho^{-1}([c, \varepsilon_2])
    \]
    defined by
    \[
    H(s, x, t) :=
    \begin{cases}
        (x, t), & \text{if } t \le \varepsilon_1, \\[4pt]
        h^{-1}(h(x, t),\, \Phi_s(t)), & \text{if } t > \varepsilon_1 ,
    \end{cases}
    \]
    is a strong deformation retract from $\rho^{-1}([c, \varepsilon_2])$ onto 
    $\rho^{-1}([c, \varepsilon_1])$.
\end{enumerate}
By similar arguments, one can construct a strong deformation retract from  $p^{-1}([\varepsilon_1, d])$ onto $p^{-1}([\varepsilon_2, d])$ for every $d \geq \varepsilon_2$.
\end{Observation}

\medskip

In the sequel, we assume that the underlying o-minimal structure is polynomially bounded. We will need the following results.
\begin{thm}[\L ojasiewicz inequality \cite{Dries-Miller96}]\label{them_loj}
Let $f, g : U \subset \mathbb{R}^n \to \mathbb{R}$ be continuous definable functions on a compact definable set $U$. Suppose that 
\[
\{x \in U : g(x) = 0\} \subset \{x \in U : f(x) = 0\}.
\]
Then, there exist constants $C > 0$ and $\theta > 0$ such that
\[
|g(x)| \ge C |f(x)|^{\theta} \quad \text{for all } x \in U.
\]
\end{thm}

The following result is due to Valette \cite{Valette2021} (see also \cite{Valette_book}). 
Although the original statement was proved in the subanalytic setting, the same argument applies to any polynomially bounded o-minimal structure.

\begin{thm}[Lipschitz conic structure \cite{Valette2021}]
\label{thm_Valette_retraction}
Let $X \subset \mathbb{R}^n$ be a definable set and $x_0 \in X$. 
For $\varepsilon > 0$ small enough, there exists a Lipschitz definable homeomorphism
\[
H : x_0 * (S(x_0, \varepsilon) \cap X) \longrightarrow \overline{B}(x_0, \varepsilon) \cap X,
\]
satisfying $H|_{S(x_0, \varepsilon) \cap X} = \mathrm{Id}$, preserving the distance to $x_0$, 
and having the following metric properties:
\begin{enumerate}    
   \item  The natural retraction by deformation onto $x_0$
    \[
    r : [0,1] \times \overline{B}(x_0, \varepsilon) \cap X \longrightarrow 
    \overline{B}(x_0, \varepsilon) \cap X,
    \]
    defined by
    \[
    r(s, x) := H\big(s H^{-1}(x) + (1-s)x_0\big),
    \]
    which is Lipschitz. Indeed, there exists a constant $C$ such that for every fixed 
    $s \in [0,1]$, the mapping $r_s$ defined by $x \mapsto r(s,x)$ is $Cs$-Lipschitz.
    
   \item  For each $\delta > 0$, the restriction of $H^{-1}$ to 
    $\{ x \in X : \delta \le \|x - x_0\| \le \varepsilon \}$ is Lipschitz and, 
    for each $s \in (0,1]$, the map 
    \[
    r_s^{-1} : \overline{B}(x_0, s\varepsilon) \cap X \longrightarrow 
    \overline{B}(x_0, \varepsilon) \cap X
    \]
    is Lipschitz.
\end{enumerate}
\end{thm}

From now on, in Sections \ref{section3}, \ref{section4} and \ref{section5} we fix a polynomially bounded o-minimal structure, and by “definable” we mean definable within this structure.

\section{Lipschitz normal embeddings}\label{section3}

Let $X \subset \mathbb{R}^n$ be a connected definable set.  
For points $x, y \in X$, the \emph{outer distance} is defined by
\[
d_{\mathrm{out}}(x,y) = \|x - y\|.
\]

The \emph{inner distance} between $x$ and $y$ is given by
\[
d_{\mathrm{in}}(x,y) = \inf_{\gamma} \length(\gamma),
\]
where the infimum is taken over all rectifiable curves $\gamma \subset X$ joining $x$ and $y$.

We say that $X$ is \textit{Lipschitz normally embedded} (or LNE for short) if there exists a constant $C \geq 1$ such that
\[
d_{\mathrm{in}}(x,y) \le C\, d_{\mathrm{out}}(x,y)
\]
for all $x, y$ in $X$.  Intuitively, this means that the intrinsic metric of $X$ and the ambient Euclidean metric are equivalent. 

Let $x_0 \in \overline{X}$. The set $X$ is said to be LNE at a point $x_0 \in \overline{X}$ (or the germ $(X, x_0)$ is LNE)  if there is a small neighborhood $U$ of $x_0$ in $\bb R^n$ such that $U\cap X$ is LNE.  

It is still unknown whether the inner metric $d_{\mathrm{in}}$ is itself definable.  
However, Kurdyka and Orro~\cite{Kurdyka1997} proved that there exists a definable metric on $X$, denoted by $\tilde{d}$, which is equivalent to $d_{\mathrm{in}}$.  
This allows one to verify the LNE condition for $X$ using $\tilde{d}$ instead of $d_{\mathrm{in}}$.

Another important distance is the \emph{diameter distance}, denoted by $d_{\diam}(x,y)$, defined as the infimum of the diameters of a  curve in $X$ joining $x$ and $y$. It was proved in \cite[Theorem~3.2]{Sampaio2025} that, in o-minimal structures on $\mathbb{R}$ (not necessarily polynomially bounded), the distances $d_{\rm in}$ and $d_{\diam}$ are equivalent. As a consequence, we have $\tilde{d}$ and $d_{\diam}$ are equivalent.

\begin{thm}\label{thm_lne}
Let $(X, 0) \subset \mathbb{R}^n$ and $(Y, 0) \subset \mathbb{R}^m$ be definable germs. Suppose $(X, 0)$ is LNE. There exists $\alpha_0 \in (0,1)$ such that if $(X, 0)$ and $(Y, 0)$ are bi-$\alpha$-H\"older equivalent for some $\alpha \in [\alpha_0, 1]$ then $(Y, 0)$ is also LNE.
\end{thm}

\begin{proof}
We will prove by contradiction that if $Y$ is not LNE then it cannot be bi-$\alpha$-H\"older homeomorphic to $(X, 0)$ when $\alpha$ is sufficiently close to $1$.

Indeed, suppose $(Y,0)$ is not LNE,   
by the \L ojasiewicz inequality (see Theorem \ref{them_loj}), there exist constants $C > 0$ and $0 < \beta < 1$ such that
\begin{equation}\label{eq_0}
    \tilde{d}(x, y) \geq C\,\|x - y\|^{\beta}
\end{equation}
for all $x, y \in Y$ sufficiently close to the origin.

Let $\varphi : (X,0) \to (Y,0)$ be an bi-$\alpha$-H\"older homeomorphism.  
For points $x, y \in Y$ near $0$, set $x' := \varphi^{-1}(x)$ and $y' := \varphi^{-1}(y)$. Since $\tilde{d}$ and $d_{\diam}$ are equivalent, we can choose a curve $\gamma_{x'y'}$ in $X$ joining $x'$ and $y'$ such that 
\[
\diam(\gamma_{x'y'}) \sim \tilde{d}(x', y').
\]
Note that $\varphi(\gamma_{x'y'})$ is a curve in $Y$ joining $x$ and $y$. Then, 
\begin{equation}\label{eq_1}
    \tilde{d}(x, y) \lesssim \diam(\varphi(\gamma_{x'y'}))
    \lesssim \big(\diam(\gamma_{x'y'})\big)^{\alpha} \sim \tilde{d}(x',y')^{\alpha}.
\end{equation}  
Since $(X,0)$ is LNE, 
\begin{equation}\label{eq_2}
     \tilde{d}(x', y')^{\alpha}
    \sim \|x' - y'\|^{\alpha}
    \lesssim  \|x - y\|^{\alpha^2}.
\end{equation}
Combining \eqref{eq_1} and \eqref{eq_2} gives
\begin{equation}\label{eq_3}
    \tilde{d}(x, y) \lesssim \|x - y\|^{\alpha^2}.
\end{equation}
When $\alpha$ is close enough to $1$ we have  $\beta < \alpha^2 <1$,  inequality \eqref{eq_3} then contradicts \eqref{eq_0}. The proof is complete. 
\end{proof}

\section{Dimension of tangent cones}\label{section4}
In this section, we prove that the dimension of the tangent cones of given definable germs must be the same if they are   bi-$\alpha$-H\"older homeomorphic when $\alpha$ is sufficiently close to $1$. The proof is inspired by the arguments of Koike and Paunescu~\cite{Koike2009} in the bi-Lipschitz case.

\begin{defn}\rm 
    Let $X\subset \bb R^n$ be a definable set. 
    \begin{itemize}
        \item The \textit{dimension of $X$} is defined by
    $$ \dim X =  \sup\{ \dim S: S \text{ is a }  C^1 \text{ submanifold of } \bb R^n \text{ contained in } X \}.$$
    \item Let $x_0 \in \overline{X}$. The dimension of the germ $(X, x_0)$, denoted by  $\dim (X, x_0)$,  is the dimension of $X\cap B(x_0, r)$ where  $r$ is sufficiently small. 
    \end{itemize}
\end{defn}

\begin{defn}\rm Let $X \subset \bb R^n$ be a set and let $x_0 \in \overline{X}$. The \textit{tangent cone of $X$} at $x_0$ is defined as
$$C_{x_0}(X) : =  \{ v\in \mathbb{R}^n: \exists (x_k)\subset X, (x_k) \to x_0,  \exists (\lambda_k) \subset \mathbb{R}_+,   \lambda_k x_k \to v\}.$$
\end{defn}
It is known that if $X$ is definable then  $C_{x_0} (X)$ is also definable, and moreover  $\dim C_{x_0}(X) \leq \dim (X, x_0)$.

\begin{defn}[Cone with vertex at the origin]\rm 
A subset $A \subset \bb R^n$ is called a \emph{cone with vertex at the origin}
(or simply a \emph{cone at $0$}) if for every $x \in A$ and every $t \geq 0$, we have
$tx \in A$.
\end{defn}

\begin{defn}[Sea-tangle neighborhood]\rm
    Let $X \subset \bb R^n$ with $0 \in \overline{X}$, and let $d, C>0$. 
    
    \begin{itemize}
    \item The \textit{sea-tangle neighborhood} $ST_d(X,C)$ of $X$ of degree $d$ and width $C$ is defined by:
    $$ST_d(X,C):=\{x \in \bb R^n: \dist(x,X) \leq C|x|^d\}.$$
    \item If $X\subset \bb R^n$ is a germ at $0$ then  we consider $ST_d(X, C)$ as a germ at $0$. 
    \end{itemize}

\end{defn}
\begin{defn}[ST equivalence]\rm 
 Let $X, Y$ be subsets of $\bb R^n$  such that $0 \in \overline{X}\cap \overline{Y}$. We say that $X$ and $Y$ are \textit{ST-equivalent}, if there are $d_1, d_2>1$, $C_1, C_2>0$ such that $Y \subset ST_{d_1}(X, C_1)$ and $X \subset ST_{d_2}(Y,C_2)$ as germs at $0 \in \bb R^n$. We write $X \sim_{ST} Y$.
 \end{defn}

\begin{lem}\label{lem_4}
    Let $X \subset \mathbb{R}^n$ be a germ at $0$. Let $\varphi: (\bb R^n, 0)\to (\bb R^n, 0)$  be  a bi-$\alpha$-H\"older homeomorphism. For $d>1$, $C>0$,  there exists $C_1, C_2 >0$ such that  
    $$ ST_{d/\alpha^2} (\varphi(X), C_1) \subset \varphi(ST_d(X, C)) \subset ST_{d\alpha^2} (\varphi(X), C_2).$$

\end{lem}
\begin{proof}
Since $\varphi$ is a bi-$\alpha$-H\"older homeomorphism, there exists a constant $K \geq 1$ such that
$$
\frac{1}{K}\|x\|^{1/\alpha} \le \|\varphi(x)\| \le K\|x\|^{\alpha}.
$$
Let $x \in ST_d(X, C)$. Then there exists $y \in X$ such that
$$
\|x - y\| \le C\|x\|^d.
$$
Applying the homeomorphism $\varphi$, we obtain
$$
\dist(\varphi(x), \varphi(X)) \le \|\varphi(x) - \varphi(y)\|
   \leq K\|x - y\|^{\alpha}
   \leq K C^\alpha\|x\|^{\alpha d}
    \leq K^{d\alpha^2 + 1} C^\alpha\|\varphi(x)\|^{\alpha^2 d}.
$$
Set $C_2 :=K^{d\alpha^2 + 1} C^\alpha$. Then
$$
\varphi(x) \in ST_{\alpha^2 d}(\varphi(X), C_2),
$$
that is,
$$
\varphi(ST_d(X, C)) \subset ST_{\alpha^2 d}(\varphi(X), C_2).
$$
Similarly, one can show that there is $C_1>0$ such that 
$$
ST_{d / \alpha^2}(\varphi(X), C_1) \subset \varphi(ST_d(X, C)).
$$
\end{proof}

\begin{lem}[{\cite[Theorem 4.14]{Koike2009}}]\label{lem_5}
    Let $X \subset \bb R^n$ be a definable germ at $0$. Then, $C_0(X)$ and $X$ are ST-equivalent.
\end{lem}

\begin{lem}[{\cite[Corollary 4.12]{Koike2009}}]\label{lem_6} Let $X ,Y\subset \bb R^n$ be germs at $0$. Then, if $X$ and $Y$ are ST-equivalent then $C_0 (X) = C_0(Y)$.    
\end{lem}

\begin{lem}[{\cite[Theorem 7.4]{Koike2009}}]\label{lem_8}
     Let $X, Y \subset \mathbb{R}^n$ be germs at $0$. Suppose that $X$ and $Y$ are $ST$-equivalent. Then, there is $d_1 > 1$ such that for any $C_1, C_2 > 0$ and for any $d$ with $1 < d \le d_1$:  
$$
\vol_n(ST_d(X, C_1) \cap  B(0, r)) \sim \vol_n(ST_d(Y, C_2) \cap  B(0, r))
$$ with $r$ sufficiently small. 
\end{lem}

\begin{lem}\label{lem_vol}
Let $A\subset\mathbb{R}^n$ be a definable cone at $0$.
Then, for every $d>1$ and $C>0$ there exists $r_0>0$ such that for all $r\in(0,r_0]$,
\[
\vol_n\big(ST_d(A,C)\cap B(0,r)\big) \sim \ r^{n+(d-1)(n-a)}.
\]
where $a = \dim A$ and the constant for "$\sim$" is independent of $r$. 
\end{lem}

\begin{proof}
Write points in polar form $x=\rho s$ with $\rho=\|x\|\in (0,\infty)$ and $s\in S^{n-1}$. 
By the polar coarea formula,
\[
\vol_n\big(ST_d(A,C)\cap B(0,r)\big)
=\int_0^r \rho^{\,n-1}\ \vol_{n-1}\big(E_\rho\big)\,d\rho,
\]
where 
\[
E_\rho:=\{\,s\in S^{n-1}:\ \rho s\in ST_d(A,C)\}.
\]

Let $S_A:=A\cap S^{n-1}$. Since $A$ is a cone, for all $\rho>0$ and $s\in S^{n-1}$,
\[
\operatorname{dist}(\rho s,A)=\rho\,\operatorname{dist}(s,A).
\]
Hence
\[
\rho s\in ST_d(A,C)\quad\Longleftrightarrow\quad
\operatorname{dist}(s,A)\le C\,\rho^{\,d-1}.
\tag{$\ast$}
\]
Thus, 
$$ E_\rho =  \{ s \in S^{n-1}: \dist(s, A) \leq C \rho^{d-1}\}.$$
Let us denote  $d_{S^{n-1}}(s, S_A)$ the spherical distance to $S_A$.
On $S^{n-1}$, in a small neighborhood $U$ of $S_A$,  $\dist(s, A) \sim d_{S^{n-1}}(s, S_A)$.
Therefore, there exist constants $\kappa_1,\kappa_2>0$ and $\rho_0>0$ such that, for all $0<\rho\le \rho_0$,
\[
\{\,s\in S^{n-1}:\ d_{S^{n-1}}(s, S_A)\le \kappa_1\,\rho^{\,d-1}\,\}
\ \subset\ 
E_\rho
\ \subset\ 
\{\,s\in S^{n-1}:\ d_{S^{n-1}}(s, S_A)\le \kappa_2\,\rho^{\,d-1}\,\}.
\tag{$\ast\ast$}
\]

We now estimate the measure of tubular neighborhoods of $S_A$ in $S^{n-1}$.
The set $S_A$ is a definable subset of $S^{n-1}$ of dimension $a-1$.
We fix a finite stratification of $S_A$ into $C^1$ embedded submanifolds of
$S^{n-1}$, called \emph{strata}.
Let $\Gamma$ be one of these strata.
By Weyl’s tube formula (see, for example, \cite{Weyl39}), for $\varepsilon>0$
sufficiently small,
\[
\vol_{n-1}\!\left(
\left\{ s \in S^{n-1} \mid d_{S^{n-1}}(s,\Gamma) \le \varepsilon \right\}
\right)
\sim
\varepsilon^{\,n-1-\dim \Gamma}\,
\vol_{\dim \Gamma}(\Gamma)
\sim
\varepsilon^{\,n-1-\dim \Gamma}.
\]
In particular, only the strata of maximal dimension $a-1$ contribute to the
leading term of the volume of small tubular neighborhoods, while strata of
lower dimension give rise to higher-order terms.
Thus, there exist constants $c_1,c_2,\varepsilon_0>0$ such that, for all $0<\varepsilon\le \varepsilon_0$,
\[
c_1\,\varepsilon^{\,n-a}
\ \le\ 
\vol_{n-1}\big(\{\,s\in S^{n-1}:\ d_{S^{n-1}}(s, S_A)\leq \varepsilon\,\}\big)
 \leq
c_2\,\varepsilon^{\,n-a}.
\tag{$\dagger$}
\]
Combining \((\ast\ast)\) with \((\dagger)\) and substituting $\varepsilon=\kappa_i\,\rho^{\,d-1}$ gives, for $0<\rho\le \rho_0$,
\[
\vol_{n-1}(E_\rho)\ \sim \ \rho^{\,(d-1)(n-a)}.
\]
Hence, for $0<r\le r_0:=\rho_0$,
\[
\vol_n\big(ST_d(A,C)\cap B(0,r)\big)
=\int_0^r \rho^{\,n-1}\,\vol_{n-1}(E_\rho)\,d\rho
\sim 
\int_0^r \rho^{\,n-1+(d-1)(n-a)}\,d\rho \sim \ r^{\,n+(d-1)(n-a)}.
\]
This completes the proof.
\end{proof}

\begin{lem}\label{lem_compare_volumes_r_rho} Let $X \subset \bb R^n$ be a definable germ at  $0$ with $\dim (X, 0) = k$. There is $r_0>0$ such that for all $0< r< \rho \leq  r_0$, we have 
$$ \vol_k \big(X \cap B(0, r)\big) \sim \bigg(\frac{\rho}{r}\bigg)^k\vol_k \big(X \cap B(0, r)\big),$$
\end{lem}
where the constant for "$\sim$" is independent of $r$ and $\rho$. 
\begin{proof}

Take $r_0 >0$ sufficiently small such that Theorem \ref{thm_Valette_retraction} holds for all $\varepsilon< r_0$ (we may replace $X$ by $X \cup \{0\}$ if necessary). 
Let $\rho = \varepsilon < r_0$ and 
let  \[
   r_s(x)=  r(s, x) := H\big(s H^{-1}(x) + (1-s)x_0\big),
    \] be the retraction defined in Theorem \ref{thm_Valette_retraction}. 
Since $x_0 = 0$, we have 
\[
r_s(x)  = H(s \cdot H^{-1}(x)),
\]
and hence the inverse map of $r_s$ is
\[
r_s^{-1}(y) = H\big((1/s)\, H^{-1}(y)\big)
\qquad \text{for } y \in X\cap B(0,s\rho).
\]
For $y_1,y_2 \in X\cap B(0,s\rho)$, we have
$$
\begin{aligned}
\|r_s^{-1}(y_1) - r_s^{-1}(y_2)\|
&= \big\| H\big( (1/s)\, H^{-1}(y_1) \big) - H\big( (1/s)\, H^{-1}(y_2) \big) \big\| \\
&\le \operatorname{Lip}(H)\cdot \big\| (1/s)\, H^{-1}(y_1) - (1/s)\, H^{-1}(y_2) \big\| \\
&= \frac{\operatorname{Lip}(H)}{s}\, \big\| H^{-1}(y_1) - H^{-1}(y_2) \big\| \\
&\le \frac{\operatorname{Lip}(H)\operatorname{Lip}(H^{-1})}{s}\, \|y_1 - y_2\| = \frac{C}{s} \|y_1 - y_2\|.
\end{aligned}
$$
where $C = \Lip (H) \Lip(H^{-1})$ which is independent of $s$. Here, $\Lip(H)$ and $\Lip(H^{-1})$ denote the Lipschitz constants of $H$ and $H^{-1}$ respectively.  
This shows that the Lipschitz constant of $r_s^{-1}$ is  $\lesssim \frac{1}{s}$. 

With $s = r/\rho$ we have $r_s\big(X \cap B(0, \rho)\big) = X \cap B(0, r)$. Equivalently, 
$$ r_s^{-1}\big(X \cap B(0, r)\big) = X \cap B(0, \rho).$$
Since $r_s^{-1}$ is Lipschitz map with constant $\lesssim 1/s$, this gives
\begin{align*}
   &  \vol_k\bigg (r_s^{-1} \big( X \cap B(0, r)\big)\bigg) \lesssim \bigg(\frac{1}{s}\bigg)^k \vol_k \big(X \cap B(0, r)\big)\\
   \Longleftrightarrow & \vol_k\bigg(X \cap B(0, \rho)\bigg)  \lesssim \bigg(\frac{\rho}{r}\bigg)^k \vol_k \big(X \cap B(0, r)\big). 
\end{align*} 
\end{proof}

\begin{lem}\label{lem_3} 
Let $X\subset\mathbb{R}^n$ be a definable germ at $0$ with $\dim (X, 0) = n$. 
Let
\[
h:(\mathbb{R}^n,0)\to(\mathbb{R}^n,0)
\]
be a bi-$\alpha$-H\"older homeomorphism with $\alpha \in (0, 1]$.
Then, there exists a constant $r_0>0$ such that for every $0<r<r_0$,
$$
\vol_n \big(h(X)\cap B(0,r)\big)
\lesssim 
r^{(\frac{1}{\alpha}+1)(\alpha-1)n} \vol_n\big( X\cap B(0,r)\big),
$$
where the constant for "$\lesssim$" is independent of $r$. 
\end{lem}

\begin{proof} 
Let $L\geq 1$ be a H\"{o}lder constant of $h$. 
For $r>0$ small define
$$
\rho=\rho(r):=\sup\{\,|x|:\ |h(x)|\le r\,\}.
$$
Then,  $h^{-1}\big(B(0,r)\big)\subset B(0,\rho)$. Put
$$
E:=X\cap B(0,\rho).
$$
Fix a sufficiently small $r$ and put $D:=\diam(E)>0$. Let $\varepsilon>0$ be arbitrary. By outer regularity of Lebesgue measure and compactness of $\overline{E}$, there exists a finite family of closed axis-parallel cubes $\{Q_i\}_{i=1}^N$ of common side length $t$ with
\[
0<t\le D,
\qquad
E \subset \bigcup_{i=1}^N Q_i,
\qquad
\sum_{i=1}^N \vol_n (Q_i) \le \vol_n (E) + \varepsilon.
\]
Each cube $Q_i$ has side $t$ and therefore  $\diam(Q_i)=t\sqrt{n}$. Denote
\[
\omega_n :=\vol_n \big(B(0,1)\big)
\]
the volume of the unit ball in $\mathbb{R}^n$. By the H\"{o}lder bound,
\[
\diam\big(h(Q_i)\big)\le L\big(\diam Q_i\big)^{\alpha}
= L(\sqrt{n}\,t)^{\alpha}.
\]
 By Jung's theorem,  $h(Q_i)$ is contained in a ball of radius $2^{-1/2}\diam(h(Q_i))$, 
\[
\vol_n \big(h(Q_i)\big)\le
\omega_n \Big(\frac{\diam(h(Q_i))}{2^{-1/2}}\Big)^n
\le \omega_n 2^{-n/2} L^n n^{\frac{\alpha n}{2}} t^{\alpha n}.
\]
Set the constant
\[
C_1 := \omega_n 2^{-n} L^n n^{\frac{\alpha n}{2}}.
\]
Summing over $i$ gives
\[
\vol_n \big(h(E)\big)
\;\le\;
\sum_{i=1}^N \vol_n \!\big(h(Q_i)\big)
\;\le\;
C_1\, N\, t^{\alpha n}.
\]
From the measure bound for the cubes we have
\[
N t^n = \sum_{i=1}^N \vol_n (Q_i) \le \vol_n (E) + \varepsilon.
\]
Hence
\[
N t^{\alpha n} = (N t^n)\, t^{(\alpha-1)n}
\le (\vol_n (E)+\varepsilon)\, t^{(\alpha-1)n}.
\]
Combining with the previous estimate yields
\[
\vol_n \big(h(E)\big)
\le C_1(\vol_n (E)+\varepsilon)\, t^{(\alpha-1)n}.
\]
Since $t\le D$, we have 
\[
\vol_n \big(h(E)\big)
\le C_1(\vol_n (E)+\varepsilon)\, D^{(\alpha-1)n}.
\]
Letting $\varepsilon \to 0$ we obtain
\[
\vol_n \big(h(E)\big)
\le C_1\, D^{(\alpha-1)n}\,\vol_n (E).
\]
Since $D \leq 2\rho$, take $C' := 2^{(\alpha-1)n}C_1$ and obtain
\begin{equation}\label{eqq0}
    \vol_n \big(h(X)\cap B(0,r)\big) \leq \vol_n(h(E))
\leq C'\,\rho^{(\alpha-1)n}\vol_n\big(X\cap B(0,\rho)\big).
\end{equation}

To finish the proof, we have to compare $\vol_n\big(X\cap B(0,\rho)\big)$ and $\vol_n\big(X\cap B(0,r)\big)$. Note from the definition that $(1/L) r^{1/\alpha}\leq \rho \leq L r^{\alpha}$. Since $\alpha - 1< 0$, the left-hand-side inequality implies that 
\begin{equation}\label{eqq01}
    \rho^{(\alpha-1)n} \leq (1/L)^{(\alpha-1)n} r^{\frac{1}{\alpha} (\alpha-1)n}.
\end{equation}
and the right-hand-side one gives
\begin{equation}\label{eqq011}
\frac{\rho}{r} \leq L r^{\alpha-1}
\end{equation}
\textbf{Case 1: } $r \geq \rho$. It is clear that $\vol_n\big(X\cap B(0, \rho)\big) \leq \vol_n\big(X\cap B(0, r)\big)$. Hence,  
$$ \rho^{(\alpha-1)n} \vol_n\big(X\cap B(0, \rho)\big)\lesssim  r^{\frac{1}{\alpha} (\alpha-1)n}\vol_n\big(X\cap B(0, r)\big). $$
Combining with \eqref{eqq0} and by $(\alpha - 1)<0$, we get 
$$ \vol_n \big(h(X)\cap B(0,r)\big)
\lesssim r^{\frac{(\alpha-1)n}{\alpha}}\vol_n\big(X\cap B(0,r)\big) \lesssim r^{(\frac{1}{\alpha}+1)(\alpha-1)n}\vol_n\big(X\cap B(0,r)\big).$$
\textbf{Case 2: }$r <\rho$. By Lemma \ref{lem_compare_volumes_r_rho} and \eqref{eqq011},
$$\vol_n\big(X\cap B(0, \rho)\big) \lesssim \bigg(\frac{\rho}{r}\bigg)^n \vol_n\big(X\cap B(0, r)\big) \lesssim r^{(\alpha-1)n} \vol_n\big(X\cap B(0, r)\big).$$
Combining with \eqref{eqq0} and \eqref{eqq01} we get
$$ \vol_n \big(h(X)\cap B(0,r)\big)
\lesssim r^{(\frac{1}{\alpha}+1)(\alpha-1)n}\vol_n\big(X\cap B(0,r)\big). $$
\end{proof}

\begin{lem}\label{lem_7}
    Let $X\subset \bb R^n$ be a definable germ at $0$. Then, there is $0< \alpha_0 <1$ such that for any bi-$\alpha$-H\"older homeomorphism $\varphi: (\bb R^n, 0) \to (\bb R^n, 0)$ with $\alpha\geq \alpha_0$ we have $$\varphi(X) \sim_{ST} \varphi(C_0(X)).$$
    In particular,   $$C_0(\varphi(X)) = C_0(\varphi(C_0(X)).$$
\end{lem}
\begin{proof}
    Since $X$ is definable, by Lemma \ref{lem_5}, $X$ and $C_0(X)$ is ST-equivalent. That is, there are $d_1, d_2 >1$ and $C_1, C_2>0$ such that 
    $ X \subset ST_{d_1}(C_0(X), C_1)$ and $ C_0(X) \subset ST_{d_2}(X, C_2)$ as germs at $0$. Applying Lemma \ref{lem_4}, we get $$\varphi(X) \subset \varphi\bigg(ST_{d_1}(C_0(X), C_1)\bigg) \subset ST_{d\alpha^2} (\varphi(C_0(X)), C_1'),$$
    and 
    $$\varphi(C_0(X)) \subset \varphi\bigg(ST_{d_1}(X, C_2)\bigg) \subset ST_{d\alpha^2} (\varphi(X), C_2')$$ for some $C_1', C_2'>0$. 
Choosing $\alpha_0$ close to $1$ such that $d\alpha_0^2>1$, we have $d\alpha^2 >1$ for all $\alpha_0 \leq \alpha\leq 1$. It is obvious that $\varphi(X)$ and $\varphi(C_0(X))$ are ST-equivalent. By Lemma  \ref{lem_6}, $C_0(\varphi(X)) = C_0(\varphi(C_0(X))$. The proof is complete.
\end{proof}

\begin{lem}\label{main lemma}
Let $E \subset \mathbb{R}^n$ be a definable cone at $0$. Then, there exists $0 < \alpha_0 < 1$ such that 
if $\varphi : (\mathbb{R}^n, 0) \to (\mathbb{R}^n, 0)$ is a bi-$\alpha$-H\"older homeomorphism for some $\alpha_0 \le \alpha \le 1$ such that 
\begin{enumerate}
    \item[(i)] $F := \varphi(E)$ and $C_0(F)$ are ST-equivalent,
    \item[(ii)] $C_0(F)$ is definable,
\end{enumerate}
then
\[
\dim C_0(F) \le \dim E.
\]
\end{lem}

\begin{proof}
Since $C_0(F)$ and $F$ are germs at $0 \in \mathbb{R}^n$ and are ST-equivalent, by Lemma \ref{lem_8}, there exist $d_1 > 0$ such that   
\begin{equation}
    \vol_n\bigg(ST_d(F, C_1) \cap B(0, r)\bigg)
\sim 
\vol_n\bigg(ST_d(C_0(F), C_2) \cap B(0, r)\bigg).
\end{equation}
for any $C_1, C_2 > 0$ and any $1< d < d_1$ and $r$ sufficiently small. We now fix $1< d < d_1$. Choose $0< \alpha_0 < 1$ such that $d \alpha_0^2 >1$. This implies that for all $\alpha_0 \leq\alpha \leq 1$ we have $1< d\alpha^2 < d$.

Fix a constant $C_3>0$, from Lemma \ref{lem_4}, there is  $C_4 >0$ such that 
$$ ST_d(F, C_3) \subset \varphi(ST_{d\alpha^2}(E, C_4)).$$
It follows that 
\[
\vol_n(ST_d(F, C_3) \cap B(0, r))
    \leq 
    \vol_n(\varphi(ST_{d\alpha^2}(E, C_4)) \cap B(0, r)).
\]
By Lemma \ref{lem_3}, there exists $r_0 > 0$ such that for all $0 < r \le r_0$,
\begin{equation}
    \vol_n \bigg(\varphi(ST_{d\alpha^2}(E, C_4)) \cap  B(0, r)\bigg)
    \lesssim 
    r^{(\frac{1}{\alpha}+1)(\alpha-1)n}
    \vol_n \bigg(ST_{d\alpha^2}(E, C_4) \cap  B(0, r)\bigg).
\end{equation}

Since $E$ is a definable cone at $0$,  by Lemma \ref{lem_vol},
\[
\vol_n(ST_{d\alpha^2}(E, C_4) \cap  B(0, r))
    \sim
    r^{(\alpha^2 - 1)d(n - a)}
    \vol_n(ST_d(E, C_4) \cap  B(0, r))
\]
where $a = \dim E$. 
Combining these estimates gives
\[
\vol_n(ST_d(F, C_3) \cap B(0, r))
    \lesssim r^K \vol_n(ST_d(E, C_4) \cap  B(0, r)),
\]
where
$$
K = K(\alpha) := (\frac{1}{\alpha}+1)(\alpha - 1)n
    + (\alpha^2 - 1)d(n - a).
$$
Note that $K< 0$ and $\lim_{\alpha\to 1} K = 0$. Taking $\alpha_0$ bigger, we may assume that $K + (d-1)>0$ for all $\alpha\geq \alpha_0$. 

Observe that 
\begin{equation}\label{equ_main4}
    1  \sim 
    \frac{\vol_n(ST_d(F, C_3) \cap B(0, r))}{
         \vol_n(ST_d(C_0(F), C_3) \cap B(0, r))}
    \lesssim
    r^K 
    \frac{\vol_n(ST_d(E, C_4) \cap B(0, r))}
         {\vol_n(ST_d(C_0(F), C_3) \cap B(0, r))}.
\end{equation}
Now assume on the contradiction that $\dim C_0(F) > \dim E$. By Lemma \ref{lem_vol}, 
$$
\frac{\vol_n(ST_d(E, C_4) \cap B(0, r))}
     {\vol_n(ST_d(C_0(F), C_3) \cap B(0, r))} \sim r^{(d-1)(\dim C_0(F) - \dim E)}.
$$
It then follows from \eqref{equ_main4} that 
\begin{equation}\label{lem_main5}
    1 \lesssim r^{K + (d-1)(\dim C_0(F) - \dim E)}
\end{equation}
Since $\dim C_0(F) - \dim E \geq 1$ and $K + (d-1) >0$,  $K + (d-1)(\dim C_0(F) - \dim E)>0$. The right hand side of  \eqref{lem_main5} tends to $0$ as $r\to 0$, which is a contradiction.  Therefore,
$$
\dim C_0(F) \le \dim E.
$$
\end{proof}

\begin{thm}\label{thm_dim}
Let $X, Y \subset \mathbb{R}^n$ be definable germs at $0$. 
Then, there exists $\alpha_0 \in (0,1)$ such that if there is  a bi-$\alpha$-H\"older homeomorphism 
$\varphi : (\mathbb{R}^n, 0) \to (\mathbb{R}^n, 0)$ with $\varphi(X) = Y$ 
for some $\alpha \in [\alpha_0, 1]$, then
\[
\dim C_0(X) = \dim C_0(Y).
\]
\end{thm}

\begin{proof}
We first show that 
\[
\dim C_0(Y) \leq \dim C_0(X).
\]
Set $E := C_0(X)$ and $F := \varphi(E)$. 
We will verify that the assumptions of Lemma~\ref{main lemma} are satisfied. 
By Lemma~\ref{lem_7}, there exists $\alpha_1$ such that if  $\alpha \in [\alpha_1, 1]$ then

(i) $\varphi(X) = Y \sim_{ST} \varphi(C_0(X)) = F$

(ii) $C_0 (Y) = C_0(F)$

By Lemma~\ref{lem_5} we have  (iii) $C_0(Y) \sim_{ST} Y$. By the transitivity of the ST-equivalence relation, we obtain $F \sim_{ST} C_0(F)$. Since $Y$ is definable, by (ii), $C_0(F) = C_0(Y)$ is also definable. Obviously, all assumptions in Lemma~\ref{main lemma} are satisfied, so 
\[
\dim C_0(Y) 
    = \dim C_0(F)
    \leq \dim E
    = \dim C_0(X).
\]
Repeating the same argument for $\varphi^{-1}$, there exists $\alpha_2 \in (0, 1)$ such that if  $\alpha \in [\alpha_2, 1]$ then
\[
\dim C_0(Y) \geq \dim C_0(X).
\]
Finally, setting $\alpha_0 := \max\{\alpha_1, \alpha_2\}$, we conclude that if  $\alpha \in [\alpha_0, 1]$ then 
\[
\dim C_0(X) = \dim C_0(Y).
\]
\end{proof}

The following lemma is well known; see, for example, \cite{Sampaio2025}, where it appears in the proof of Theorem 4.1. For the reader’s convenience, we include a detailed proof here.
\begin{lem}\label{lem_extension}
Let $X$ and $Y$ be germ at $0$ in $\bb R^n$.  Assume there exists a bi-$\alpha$-H\"older homeomorphism $h:(X,0)\to (Y,0)$ with $0<\alpha\le1$. 
Then there exists a bi-$\alpha^2$-H\"older homeomorphism 
\[
\Phi:(\mathbb{R}^{2n},0)\longrightarrow (\mathbb{R}^{2n},0)
\quad\text{such that}\quad 
\Phi\bigl(X\times\{0\}\bigr)=\{0\}\times Y.
\]
\end{lem}

\begin{proof}
Write $h=(h_1,\dots,h_n)$ and $h^{-1}=(k_1,\dots,k_n)$. 
There exist $C,L\geq 1$ such that for all $x,x'\in X$ and $y,y'\in Y$,
\[
\|h(x)-h(x')\|\leq C\,\|x-x'\|^\alpha,\qquad 
\|h^{-1}(y)-h^{-1}(y')\|\leq L\,\|y-y'\|^\alpha.
\]

By the McShane–Whitney extension for H\"older functions (see for example \cite{McShane1934}), each coordinate function 
$h_i:X\to\mathbb{R}$ admits an $\alpha$-H\"older extension $\widetilde h_i:\mathbb{R}^n\to\mathbb{R}$; 
 and similarly each $k_j:Y\to\mathbb{R}$ extends to 
$\widetilde k_j:\mathbb{R}^n\to\mathbb{R}$. 
Set $\widetilde h=(\widetilde h_1,\dots,\widetilde h_n):\mathbb{R}^n\to\mathbb{R}^n$ and 
$\widetilde k=(\widetilde k_1,\dots,\widetilde k_n):\mathbb{R}^n\to\mathbb{R}^n$. 
Then $\widetilde h|_X=h$ and $\widetilde k|_Y=h^{-1}$, and there are constants 
$C',L'>\geq 1$ such that
\[
\|\widetilde h(u)-\widetilde h(u')\|\leq C'\|u-u'\|^\alpha,\qquad 
\|\widetilde k(v)-\widetilde k(v')\|\leq L'\|v-v'\|^\alpha
\quad (u,u',v,v'\in\mathbb{R}^n).
\]

Define $\Phi,\Psi:(\mathbb{R}^{2n},0)\to (\mathbb{R}^{2n}, 0)$ by
\[
\Phi(u,v):=\bigl(u-\widetilde k\bigl(v+\widetilde h(u)\bigr),\ v+\widetilde h(u)\bigr),\qquad
\Psi(z,w):=\bigl(z+\widetilde k(w),\ w-\widetilde h\bigl(z+\widetilde k(w)\bigr)\bigr).
\]

\textbf{Claim 1:  $\Psi=\Phi^{-1}$}.

For any $(u,v)\in\mathbb{R}^{2n}$ near the orgin, we have 
\begin{align*}
\Psi\bigl(\Phi(u,v)\bigr)
&=\Psi\bigl(u-\widetilde k(v+\widetilde h(u)),\, v+\widetilde h(u)\bigr)\\
&=\Bigl(u-\widetilde k(v+\widetilde h(u))+\widetilde k\bigl(v+\widetilde h(u)\bigr),\ 
\ (v+\widetilde h(u))-\widetilde h\bigl(u-\widetilde k(v+\widetilde h(u))+\widetilde k(v+\widetilde h(u))\bigr)\Bigr)\\
&=(u,\ v+\widetilde h(u)-\widetilde h(u))=(u,v).
\end{align*}
Similarly, for any $(z,w)\in\mathbb{R}^{2n}$ near the origin,
\begin{align*}
\Phi\bigl(\Psi(z,w)\bigr)
&=\Phi\bigl(z+\widetilde k(w),\, w-\widetilde h(z+\widetilde k(w))\bigr)\\
&=\Bigl(z+\widetilde k(w)-\widetilde k\bigl(w-\widetilde h(z+\widetilde k(w))+\widetilde h(z+\widetilde k(w))\bigr),\ 
 \ w-\widetilde h(z+\widetilde k(w))+\widetilde h(z+\widetilde k(w))\Bigr)\\
&=(z+\widetilde k(w)-\widetilde k(w),\ w)=(z,w).
\end{align*}
Thus $\Psi=\Phi^{-1}$ and $\Phi$ is a homeomorphism.

\medskip
\textbf{Claim 2:} $\Phi$ and $\Psi$ are bi-$\alpha^2$-H\"older.

Fix $(u,v),(u',v')\in\mathbb{R}^{2n}$ near the origin, write $\Phi =  (\Phi_1, \Phi_2) \in \bb R^n \times \bb R^n$. Using the $\alpha$-H\"older property of 
$\widetilde h$ and $\widetilde k$, 
\begin{align*}
\bigl\|\,\Phi_2(u,v)-\Phi_2(u',v')\,\bigr\|
&=\bigl\|\,v+\widetilde h(u)-v'-\widetilde h(u')\,\bigr\|\\
& \leq\ \|v-v'\|+C'\|u-u'\|^\alpha,
\lesssim \|(u,v) - (u', v')\|^{\alpha^2} \text{ since } \alpha \in (0, 1]; 
\end{align*}
and
\begin{align*}
\bigl\|\,\Phi_1(u,v)-\Phi_1(u',v')\,\bigr\|
&=\bigl\|\,u-u' -\widetilde k(v+\widetilde h(u))+\widetilde k(v'+\widetilde h(u'))\,
\bigr\|\\
&\leq \|u-u'\|+L'\,\bigl\|\, (v-v')+(\widetilde h(u)-\widetilde h(u'))\,\bigr\|^\alpha\\
&\leq \|u-u'\|+L'\bigl(\|v-v'\|+C'\|u-u'\|^\alpha\bigr)^\alpha\\
& \lesssim   \|u-u'\|+\|v-v'\|^\alpha + \|u-u'\|^{\alpha^2}.\\
& \lesssim \|(u,v) - (u', v')\|^{\alpha^2}.
\end{align*}
This implies that $\Phi$ is $\alpha^2$-H\"older. The same argument applied to $\Psi$ (symmetry of the formulas) shows that $\Psi$ is $\alpha^2$-H\"older. Therefore, $\Phi$ is bi-$\alpha^2$-H\"older.

\medskip
\textbf{Claim 3:   $\Phi(X\times\{0\}) = \{0\} \times Y$.}

Let $x\in X$. Using the extension properties $\widetilde h|_X=h$ and $\widetilde k|_Y=h^{-1}$, 
\[
\Phi(x,0)=\bigl(x-\widetilde k(\widetilde h(x)),\ \widetilde h(x)\bigr)
=\bigl(x-h^{-1}(h(x)),\ h(x)\bigr)=(0,h(x))\in\{0\}\times Y.
\]
Thus $\Phi(X\times\{0\})=\{0\}\times Y$.

\medskip
This completes the proof.
\end{proof}

\begin{thm}\label{thm_dim2}
Let $X \subset \mathbb{R}^n$ and $Y \subset \mathbb{R}^m$ be definable germs at $0$. Then there exists $\alpha_0 \in (0, 1)$ such that if $(X, 0)$ and $(Y, 0)$ are bi-$\alpha$-H\"older equivalent for some $\alpha \in [\alpha_0, 1]$, then 
\[
\dim C_0(X) = \dim C_0(Y).
\]
\end{thm}

\begin{proof}
Without loss of generality, we assume that $n \geq m$. By embedding $Y$ into $\mathbb{R}^n$, we may consider $Y$ as a germ at $0$ in $\mathbb{R}^n$. Let $h : (X, 0) \to (Y, 0)$ be a bi-$\alpha$-H\"older homeomorphism.  
By Lemma~\ref{lem_extension}, there exists a bi-$\alpha^2$-H\"older homeomorphism 
\[
\varphi : (\mathbb{R}^{2n}, 0) \to (\mathbb{R}^{2n}, 0)
\]
such that $\varphi(X \times \{0\}) = \{0\} \times Y$.  
Applying Theorem~\ref{thm_dim}, we obtain $\alpha_1 \in (0, 1)$ such that if $\alpha^2 \geq \alpha_1$, then
\[
\dim C_0(X \times \{0\}) = \dim C_0(\{0\} \times Y).
\]
Since $\dim C_0(X \times \{0\}) = \dim C_0(X)$ and $\dim C_0(\{0\} \times Y) = \dim C_0(Y)$, it follows that
\[
\dim C_0(X) = \dim C_0(Y).
\]
Hence, setting $\alpha_0 := \sqrt{\alpha_1}$ yields the desired constant satisfying the statement of the theorem.
\end{proof}

\section{Homotopy types of the links of the tangent cones}\label{section5}

In this section, we prove a bi-$\alpha$-H\"older rigidity result for the homotopy type of links of tangent cones.

\begin{lem}\label{lem_homotopy}
Let $A\subset \bb R^n$ be a closed definable cone at $0$.  
Then, for any $d > 1, C > 0$,  the inclusion  
\[
\iota : (A, 0) \hookrightarrow (ST_d(A, C), 0),
\]
induces isomorphisms between the homotopy groups of $\Lk(A, 0)$ and $\Lk(ST_d(A, C),0)$.
\end{lem}

\begin{proof}
Let $S_A:=A\cap S^{n-1}$ and define the definable continuous function
\[
g:S^{n-1}\to \mathbb{R}_{\ge 0},\qquad g(u)=\operatorname{dist}(u,A).
\]
Since $A$ is a cone, $\operatorname{dist}(ru,A)=r\,\operatorname{dist}(u,A)=r\,g(u)$ for all $r>0$, $u\in S^{n-1}$.
Hence, with $\varphi_r:S(0, r)\to S^{n-1}$, $\varphi_r(x)=x/r$, we have
\[
ST_d(A, C) \cap \bb \bb S(0, r)=\{ru: u \in S^{n-1}, \ g(u)\le C r^{d-1}\}=\varphi_r^{-1}\!\big(g^{-1}([0,Cr^{d-1}])\big).
\]

By Hardt’s triviality theorem, there exists $\delta>0$ such that $g$ is a definably trivial over $(0,\delta)$.
Using Observation \ref{Observation1} (2) we see that  for $0<\varepsilon_1<\varepsilon_2<\delta$,  there is a strong deformation retract from $g^{-1}([0,\varepsilon_2]$ onto $g^{-1}([0,\varepsilon_1])$, hence the inclusion
$g^{-1}([0,\varepsilon_1])\hookrightarrow g^{-1}([0,\varepsilon_2])$ induces isomorphisms on all homotopy groups.

Since $S_A$ is a compact definable set, it is triangulable; thus there exist arbitrarily small open neighborhood $U$ admitting a definable deformation retraction onto $S_A$.
By continuity of $g$ and compactness of $S_A$, we may choose two such neighborhoods  $U_1 \subset U_2$ and 
$
0<\varepsilon_1<\varepsilon_2<\delta$ such that 
$$
g^{-1}([0,\varepsilon_1])\subset U_1\subset g^{-1}([0,\varepsilon_2])\subset U_2.
$$

It is clear that the inclusion maps from $\iota_1: g^{-1}([0, \varepsilon_1])\to g^{-1}([0, \varepsilon_2])$ and $\iota_2: U_1 \to U_2$ induce isomorphisms between corresponding homotopy groups. 

Consider the following commutative diagram induced from inclusion maps: 

$$
\begin{tikzcd}
\pi_*(g^{-1}([0, \varepsilon_1]) \arrow[r, "\iota_{1*}"] \arrow[dr] & \pi_*(g^{-1}([0, \varepsilon_2])  \arrow[dr] & \\
\pi_*(S_A)\arrow[u]\arrow[r, "\iota_{0*}"]  & \pi_*(U_1) \arrow[u] \arrow[r, "\iota_{2*}"] & \pi_*(U_2)
\end{tikzcd}
$$

Since $\iota_{i*}$, $i = 0, 1, 2$ are isomorphisms, so all morphisms in the diagram are isomorphic. This implies that the homotopy groups of $S_A$ and $g^{-1}([0, \varepsilon])$
are isomorphic for all $\varepsilon< \delta$. 
Conjugating by the homeomorphism
$\varphi_r:S(0,r)\to S^{n-1}$ shows that
\[
A \cap \bb \bb S(0, r)=\varphi_r^{-1}(S_A)\ \hookrightarrow\ ST_d(A, C) \cap \bb \bb S(0, r) =\varphi_r^{-1}\big(g^{-1}([0,Cr^{d-1}])\big)
\]
induces isomorphisms on all homotopy groups for sufficiently small $r>0$ provided $Cr^{d-1} < \delta$.
This completes the proof.
\end{proof}

\begin{prop}\label{prop_homotopy}
Let $(X,0)$ and $(Y,0)$ be closed definable germs in $\mathbb{R}^n$.  
For every $C>0$ there exists $d_0>1$ and $\alpha_0\in (0,1)$ such that  
if $(X,0)$ and $(Y,0)$ are embedded bi-$\alpha$-Hölder equivalent for some $\alpha\in[\alpha_0,1]$,  
then for every $d\in (1, d_0)$ the links $
\Lk(ST_d(X,C),0)$ and $
\Lk(ST_d(Y,C),0)$
have the same homotopy type.
\end{prop}

\begin{proof}

Fix $C>0$.  For a definable germ $Z\subset\mathbb{R}^n$ (with $Z=X$ or $Z=Y$), set
\[
\Omega_Z := \{x\in B(0,1): \dist(x,Z)\le C\|x\|^d \text{ for some } d\in(1,2]\},
\]
and define
\[
\rho_Z(x) := \sup\{\,d\in(1,2] : \dist(x,Z)\le C\|x\|^d\,\}, \qquad x\in\Omega_Z.
\]
For $x\in\Omega_Z\setminus Z$, the function
\[
q_Z(x):=\frac{\log(\dist(x,Z)/C)}{\log\|x\|}
\]
is continuous. 
Since
\[
\rho_Z(x)=\min\{2,q_Z(x)\},
\]
$\rho_Z$ is continuous on $\Omega_Z\setminus Z$.

Although $\Omega_Z$ need not be definable in the original structure, it is definable in an exponential o-minimal expansion (see \cite{MS2002}).  In particular, $\rho_Z$ is definable on the same structure as $\Omega_Z$. Applying Hardt’s triviality theorem to $\rho_Z|_{\Omega_Z\setminus Z}$, we obtain $\delta_Z>1$ such that $\rho_Z$ is definably trivial over $(1,\delta_Z)$.  
By Observation~\ref{Observation1},  for any 
$1 < d < d' < \delta _Z$, there exist strong deformation retract from $(\rho_Z|_{\Omega_Z\setminus Z})^{-1}([d, \delta_Z]) = ST_d(Z, C)\setminus X$ onto $
(\rho_Z|_{\Omega_Z\setminus Z})^{-1}([d', \delta_Z]) = ST_{d'}(Z, C) \setminus Z$.  Since $Z \subset ST_d'(Z, C)$, we may extend this deformation retract by defining it to be the identity on $Z$. This yields a strong deformation retract from $ST_d(Z, C)$ onto $ST_{d'}(Z, C)$ that fixes $Z$, hence fixes $0$ as well. Consequently, it induces a strong deformation retract from $ST_{d}(Z, C) \setminus\{0\}$ onto $ST_{d'}(Z, C) \setminus\{0\}$ ({$\dagger$}).

Let $\varphi: (\bb R^n, 0) \to (\bb R^n, 0) $ be a bi-$\alpha$-H\"older homeomorphism such that $\varphi(X) = Y$. Set $d_0 :=  \min\{\delta_X, \delta_Y\}$ and  fix an $d \in (1, d_0)$.  It follows from Lemma \ref{lem_4} that there are $C_1, C_2 >0$ such that 
\begin{equation}
      ST_{d/\alpha^2} (Y, C_1) \subset  \varphi(ST_d(X, C)) \subset ST_{d\alpha^2} (Y, C_2)
\end{equation}
Set $d_1 = d_1(\alpha): = d/\alpha^2 + (1-\alpha) > d/\alpha^2$, $d_2  = d_2(\alpha) := d\alpha^2 - (1-\alpha) < d\alpha^2$. There is $\alpha_0 <1$ close to $1$ such that for all  $\alpha_0 \leq \alpha <1$, both $d_1$ and $d_2$ are greater than $1$. Since $ST_{d_1} (Y, C) \subset  ST_{d/\alpha^2} (Y, C_1)$  and $ST_{d\alpha^2} (Y, C_2) \subset ST_{d_2} (Y, C)$,  we have 
\begin{equation}\label{equ_3.1}
    ST_{d_1} (Y, C) \subset  \varphi(ST_d(X, C)) \subset ST_{d_2} (Y, C).
\end{equation}
Again by  Lemma \ref{lem_4} there is $C_3>0$ such that $$\varphi(ST_{d_1/\alpha^2} (X, C_3)) \subset ST_{d_1}(Y, C).$$
Choosing $d_3 = d_3(\alpha):= d_1/\alpha^2 + (1-\alpha) >  d_1/\alpha^2$, we have 
$$
    ST_{d_3} (X, C) \subset ST_{d_1/\alpha^2} (X, C_3).
$$
Thus, 
\begin{equation}\label{equ_3.2} \varphi(ST_{d_3} (X, C)) \subset ST_{d_1}(Y, C)\end{equation}
Combining \eqref{equ_3.1} and \eqref{equ_3.2} gives
$$\varphi(ST_{d_3} (X, C)) \subset  ST_{d_1} (Y, C) \subset  \varphi(ST_d(X, C)) \subset ST_{d_2} (Y, C).$$
Note that $d_3 > d_1> d_2$ all tend to $d$ as $\alpha$ tends to $1$. Take $\alpha_0$ close enough to $1$ such that for all $\alpha \in [\alpha_0, 1]$, $d_i\in (1, d_0)$ for $i = 1, 2, 3$.  Consider the following commutative diagram induced by inclusion maps: 

$$
\begin{tikzcd}
\pi_*\bigg( \varphi( ST_{d_3}(X, C) \setminus\{0\})\bigg) \arrow[r, "\iota_{1*}"] \arrow[dr] & \pi_*\bigg(\varphi(ST_{d}(X, C) \setminus\{0\})\bigg)  \arrow[dr] & \\
& \pi_*\bigg(ST_{d_1}(Y, C) \setminus\{0\}\bigg) \arrow[u] \arrow[r, "\iota_{2*}"] & \pi_*\bigg(ST_{d_2}(Y, C) \setminus\{0\}\bigg)
\end{tikzcd}
$$
Since $\iota_{i,*}$, $i = 1, 2$ are isomorphic, so are the others in the diagram. As  $\varphi$ is a homeomorphism, it follows that the homotopy groups of $ST_{d_3}(X, C) \setminus\{0\}$ and $ST_{d_1}(Y, C) \setminus\{0\}$ are isomorphic. Together with ($\dagger$), we see that the homotopy types of $ST_{d'}(X, C) \setminus\{0\}$ and $ST_{d''}(Y, C) \setminus\{0\}$ are the same for all $d', d'' \in (1, d_0)$.  It is known that $ST_{d}(Z, C) \setminus\{0\}$, $Z= X, Y$ can always deformation retracts onto $\Lk(ST_d(Z, C), 0)$. Therefore, $\Lk(ST_d(X, C), 0)$ and $\Lk(ST_d(Y, C), 0)$ are of the same homotopy type. 

Choose $\alpha_0 = \alpha_0(d)$ for any $d$. We see that $\alpha_0$ satisfies the conclusion of the proposition. The proof is complete. 
\end{proof}

By the same trick, we have
\begin{prop}\label{prop_homotopy2} Let $X\subset \bb R^n$ be a closed definable germ at $0$. For every $C>0$, there exists constant $d_0>1$ such that for every $0< d< d_0$ the homotopy types of  $\Lk(ST_d(X, C),0)$ and $\Lk(ST_d(C_0(X), C),0)$ are the same.
\end{prop}
\begin{proof}
It suffices to show that there is $d_0>1$ such that for every $1< d<d_0$, the homotopy groups of 
$ST_{d}(X,C) \setminus \{0\}$ and $ST_{d}(C_0(X),C) \setminus \{0\}$ 
are isomorphic.

Let $A$ be the germ of $C_0(X)$ as $0$ and fix $C > 0$.  
Define
\[
\Omega_X := \{x \in B(0, 1):
\dist(x, X) \le C \|x\|^d, d\in (1, 2]\},
\]
\[
\Omega_A := \{x \in B(0, 1) : 
\dist(x, A) \le C \|x\|^d, d\in (1, 2]\}.
\]
Consider the  continuous  map $\rho_X = \sup\{d\in (1, 2]: \dist(x, X) \leq C \|x\|^d\}$ and  $\rho_A = \sup\{d\in (1, 2]: \dist(x, A) \leq C\|x\|^d\}$ defined on $\Omega_X\setminus X$ and $\Omega_A\setminus A$ respectively.  
Arguing as in the proof of Proposition \ref{prop_homotopy},   we obtain a constant $d_0 > 1$ such that for any $1< d<d'< d_0$ there are strong deformation retracts from  $ST_d(X, C)\setminus\{0\}$ onto $ST_{d'}(X, C)\setminus\{0\}$ and from $ST_d(A, C)\setminus\{0\}$ onto $ST_{d'}(A, C)\setminus\{0\}$.

Since $X$ and $A$ are $ST$-equivalent, we may choose numbers 
\[
1 < d_1 < d_2 < d_3 < d_4 < d_0
\]
such that
\[
ST_{d_4}(X, C) \subset ST_{d_3}(A, C) \subset ST_{d_2}(X, C)
\subset ST_{d_1}(A, C).
\]
After removing the origin, and using the same inclusion argument used in  Proposition~\ref{prop_homotopy}, we conclude that  
the homotopy groups of 
$ST_{d}(X,C) \setminus \{0\}$ and $ST_{d}(A,C) \setminus \{0\}$ for every $1< d<d_0$
are isomorphic.  
This completes the proof.
\end{proof}

\begin{thm}\label{thm_homotopy}
 Let $(X, 0)$ and $(Y,0)$ be closed definable germs in $\mathbb R^n$. 
Then, there exists $\alpha_0\in(0,1)$ such that if $(X,0)$ and $(Y,0)$ are embedded bi-$\alpha$-H\"older equivalent for some $\alpha\in [\alpha_0, 1]$, then the homotopy types of $\Lk(C_0(X),0)$ and $\Lk(C_0(Y),0)$ are the same.
\end{thm}

\begin{proof}
Fix a constant $C>0$. By Propositions~\ref{prop_homotopy2} and \ref{prop_homotopy}, there exist 
$d_0>1$ and $\alpha_0 \in (0,1)$ such that if $(X,0)$ and $(Y,0)$ are embedded 
bi-$\alpha$-H\"older equivalent for some $\alpha \in [\alpha_0,1]$, then for every 
$d \in (1,d_0)$ the links
\[
\Lk\!\big(ST_d(C_0(X),C),0\big) 
\quad\text{and}\quad 
\Lk\!\big(ST_d(C_0(Y),C),0\big)
\]
have the same homotopy type.

By Lemma~\ref{lem_homotopy}, the homotopy groups of 
$\Lk(ST_d(C_0(X),C),0)$ and $\Lk(C_0(X),0)$ (resp. 
$\Lk(ST_d(C_0(Y),C),0)$ and $\Lk(C_0(Y),0)$) are isomorphic. 
It follows that $\Lk(C_0(X),0)$ and $\Lk(C_0(Y),0)$ have the same homotopy type.

Therefore, any such $\alpha_0$ satisfies the conclusion of the theorem.
\end{proof}

Combining Theorem \ref{thm_homotopy} and Lemma \ref{lem_extension} we get 
\begin{thm}\label{thm_homotopy2}
    Let $(X, 0) \subset \bb R^n$  and $(Y,0)\subset \bb R^m$ be closed definable germs.  
Then, there exists $\alpha_0\in(0,1)$ such that if $(X,0)$ and $(Y,0)$ are bi-$\alpha$-H\"older equivalent for some $\alpha\in [\alpha_0, 1]$, then the homotopy type of   $\Lk(C_0(X),0)$ and $\Lk(C_0(Y),0)$ are the same.
\end{thm}

Theorem~\ref{thm_homotopy2} has the following immediate consequence for singular homology (see for example \cite[Proposition 4.21]{Hatcher2002}).
\begin{cor}
Let $(X, 0) \subset \bb R^n$  and $(Y,0)\subset \bb R^m$ be closed definable germs.  
Then, there exists $\alpha_0\in(0,1)$ such that if $(X,0)$ and $(Y,0)$ are bi-$\alpha$-H\"older equivalent for some $\alpha\in [\alpha_0, 1]$, then the singular homology groups of $\Lk(C_0(X),0)$ and $\Lk(C_0(Y),0)$ are isomorphic.
\end{cor}

\section{Applications}
In \cite{Sampaio2025}, Sampaio proved that:

\begin{thm}[{\cite[Theorem 4.1]{Sampaio2025}}]
    Let $(X, 0) \subset \bb C^n$ be complex analytic germ. If   \((X, 0)\) and \((\mathbb{R}^k, 0)\) are bi-\(\alpha\)-H\"older equivalent for every \( 0 < \alpha < 1 \), then \(( X,0) \) is smooth. 
\end{thm} 
In the following we present a simple proof of this theorem in a  stronger form.

\begin{thm}
Let $(X, 0)\subset \bb C^n$ be a complex analytic germ. Then, there exists \( \alpha_0\in (0, 1)\) such that if \((X, 0)\) and $(\bb R^k, 0)$  are bi-\(\alpha\)-H\"older equivalent for some $\alpha \in [\alpha_0, 1]$, then $(X, 0)$ is  smooth.
\end{thm}

\begin{proof}
By Theorem~\ref{thm_homotopy2}, there exists $\alpha_0 \in (0, 1)$ such that if $(X, 0)$ and $(\mathbb{R}^k, 0)$ are bi-$\alpha$-H\"older equivalent for some $\alpha \geq \alpha_0$, then  $\Lk (C_0(X),0)$ and $\Lk(\mathbb{R}^k, 0)$ share the same homotopy type.  By Prill’s theorem~\cite{Prill1967}, $C_0(X)$ is a linear subspace.  Moreover, since $(\mathbb{R}^k, 0)$ is LNE, by Theorem~\ref{thm_lne}, increasing $\alpha_0$ if necessary, $(X, 0)$ is also LNE. 
Finally, by~\cite[Proposition~3.4]{BFLS2026} (see also \cite[Theorem 4.1]{FS2023}), $(X, 0)$ is smooth .
\end{proof}
\section{Counterexamples in the non-polynomially bounded case}\label{section6}

\begin{example}[{\cite[Example~4]{FEJ2018}}]\rm 
Let $X: = \{x \in \mathbb{R}\}$ and 
\[
Y := \{(x, y) \in \mathbb{R}^2 : x \neq 0,\; y = ||x|\log |x||\} \cup \{(0, 0)\}.
\]
Define the map germ $h \colon (X, 0) \to (Y, 0)$ by 
\[
h(0) := (0,0), \qquad h(x) := (x, ||x|\log (|x|)|) \text{ for } x \neq 0.
\]
Then $h$ is a bi-$\alpha$-H\"older homeomorphism for all $\alpha \in (0, 1)$. 
However,$(X, 0)$ is LNE, while $(Y, 0)$ is not.  
This provides a counterexample to Question~5.
\end{example}

\begin{example}\rm
Let 
$$
X := \{(x, y) \in \mathbb{R}^2\}$$
and $$
Y := \{(x, y, f(x, y)) \in \mathbb{R}^3 : (x, y) \neq (0, 0),\; f(x, y) := |r \log (r)|\text{ where } r = \|(x, y)\| \} \cup \{(0, 0, 0)\}.
$$
Consider the homeomorphism $h \colon (X, 0) \to (Y, 0)$ defined by 
$$
h(0, 0) := (0, 0, 0), \qquad h(x, y) := (x, y, f(x, y)) \text{ for } (x, y) \neq (0, 0).
$$
Then the germ $h$ is bi-$\alpha$-H\"older for all $\alpha \in (0, 1)$. 
However,
\begin{itemize}
    \item $\dim C_0(X) = 2 \neq \dim C_0(Y) = 1$,
    \item $\pi_1(\Lk(C_0(X), 0)) = \bb Z$ while $\pi_1(\Lk( C_0(Y), 0)) = 0$.
\end{itemize}
This provides a counterexample to Questions~1–4.
\end{example}
\bibliographystyle{siam}
\bibliography{Biblio2}

\end{document}